\documentclass[11pt,reqno]{amsart}
\usepackage{hyperref}
\usepackage{amsmath, amssymb, latexsym, amscd, amsthm, amsfonts, amstext}
\usepackage[mathscr]{eucal}
\usepackage{xspace}

\DeclareMathOperator*{\einf}{ess\ inf}

\voffset= - 0.5in
\theoremstyle{plain}
\newtheorem{theorem}{Theorem}[section]

\newtheorem{lemma}[theorem]{Lemma}
\newtheorem{remark}[theorem]{Remark}
\newtheorem{proposition}[theorem]{Proposition}

\theoremstyle{definition}
\newtheorem{definition}[theorem]{Definition}

\begin{document}
\title[Nonlocal singular elliptic problems in Musielak-Sobolev spaces]{Positive Ground State Solutions to a Nonlocal Singular Elliptic Problem}
\author[Mustafa Avci]{Mustafa Avci}
\address{Mustafa Avci \\
Department of Finance and Management Science, Edwards School of Business, University of Saskatchewan, Canada}
\email{{\href{mailto: M. Avci <avcixmustafa@gmail.com>}{avcixmustafa@gmail.com}} (Primary), {\href{mailto: M. Avci <avci@edwards.usask.ca>}{avci@edwards.usask.ca}}}

\keywords{Nonlocal elliptic problem, singular problem, Musielak-Orlicz-Sobolev spaces \\
\hspace*{.3cm} \textit{2000 Mathematics Subject Classifications}: 35J40; 35J60; 35J70}

\begin{abstract}
In the present paper, we study the existence and uniqueness of solutions to some nonlocal singular elliptic problem under Dirichlet boundary condition. Problem is settled in Musielak-Sobolev spaces.
\end{abstract}

\maketitle

\centerline{}

\numberwithin{equation}{section}

\section{Introduction}
In this article, we are concerned with a nonlocal singular elliptic problem of the form
\begin{equation}
\begin{cases}
\begin{array}{rlll}
-\mathcal{A}\left(\int_{\Omega}\Phi(x,|\nabla u|)dx\right)\mathrm{div}(a(x,|\nabla u|)\nabla u)& = & g(x)u^{-\gamma(x)} \text{ in }\Omega, \\
u & > & 0 \text{ in }\Omega,\\
u & = & 0 \text{ on }\partial \Omega,
\end{array}
\end{cases}
\label{e1.1}
\end{equation}
where $\Omega$ is a bounded domain in $\mathbb{R}^N$ $(N\geq 3)$ with smooth boundary $\partial\Omega$, $\gamma:\overline{\Omega}\rightarrow (0,1)$ is a continuous function, $\mathcal{A}$ and $g$ are continuous functions.\\
The function $\varphi(x,t):=a(x,|t|)t$ is an increasing homeomorphism from $\Omega\times\mathbb{R}$ onto $\mathbb{R}$. We want to remark that if we let $a(x,t)=|t|^{p(x)-2}$, where $p(x)$ is a continuous function on $\overline{\Omega}$ with $\inf_{x \in \overline{\Omega}}p(x)>1$, equation \ref{e1.1} turns into the well-known singular $p(x)$-Kirchhoff equation. If we additionally consider the case $\mathcal{A}(t)=1$, equation \ref{e1.1} becomes singular the $p(x)$-Laplace equation, a generalization of $p$-Laplace equation, given by $\mathrm{div}(|\nabla u|^{p(x)-2}\nabla u)=f(x,u)$, $1<p(x)<N$. Therefore, equation \ref{e1.1} particularly generalizes the problems involving variable exponent. This kind of equations have been intensively studied by many authors for the past two decades due to its significant role in many fields of mathematics, such as in the study of calculus of variations, partial differential equations \cite{Acerbi,DCU,Diening}, but also for their use in a variety of physical and engineering contexts: the modeling of electrorheological fluids \cite{Ruzicka}, the analysis of Non-Newtonian fluids \cite{Zhikov}, fluid flow in porous media \cite{Amaziane}, magnetostatics \cite{Cekic}, image restoration \cite{Blomgren}, and capillarity phenomena \cite{Avci0}, see also, e.g., \cite{Benali,Alves,AS,Avci6,Avci4,Bonanno1,Boureanu2,Chung2,Heidarkhani,Yuce2} and references therein. Therefore, equation \eqref{e1.1} may represent a variety of mathematical models corresponding to certain phenomena:\\
For $\varphi(t):=p|t|^{p-2}t$;
\begin{itemize}
\item Nonlinear elasticity: $\varphi (t)=\left(1+t^{2}\right) ^{\alpha }-1,$ $\alpha >\frac{1}{2},$
\item Plasticity: $\varphi (t)=t^{\alpha}\left(\log\left( 1+t\right)\right)^{\beta},$ $\alpha \geq 1,\beta >0,$
\item Generalized Newtonian fluids: \ $\varphi(t)=\int_{0}^{t}s^{1-\alpha }\left(\sinh^{-1}s\right)^{\beta}ds,$ \\ $0\leq\alpha\leq 1,\beta>0.$
\end{itemize}
For $\varphi(t)=\varphi(x,t):=p(x)|t|^{p(x)-2}t$;
\begin{itemize}
\item There is a new model for image restoration given in \cite{CLR}. In this model, main aim is to recover an image, $u$, from an observed, noisy image, $u_{0}$, where the two are related by $u_{0}=u+noise$. The proposed model incorporates the strengths of the various types of diffusion arising from
the minimization problem
$$
E(u)=\int_{\Omega }\left[ \left\vert \nabla u\right\vert ^{p\left( x\right)}+\lambda \left( u-u_{0}\right) ^{2}\right] dx
$$
for $1\leq p\left( x\right) \leq 2$, where $\int_{\Omega }\left\vert \nabla u\right\vert ^{p\left( x\right) }dx$ is a regularizing term to remove the noise and $\lambda \geq 0$.
\end{itemize}
Recently, there has been some paper dealing with nonlocal singular problems see, e.g, \cite{Lei,Liao,Mu,Saodi,Yijing} and references therein. However, to the best knowledge of the author, problem \eqref{e1.1} is not covered in the literature.

\section{Preliminaries}
We use the theory of Orlicz spaces since problem \ref{e1.1} contains a nonhomogeneous function $\varphi$ in the nonlinear differential operator $\mathrm{div}(a(x,\cdot))$. Therefore, we start with some basic concepts of Orlicz spaces. For more details, we refer the readers to the monographs \cite{Adams,Krasnosels'kii,Musielak,RenRao}, and to the papers \cite{XLFan1,Petteri,Hudzik,MihaRadu}.

The function $a(x,t): \Omega\times \mathbb{R}\rightarrow(0,\infty)$ is a function such that the mapping $\varphi(x,t): \Omega \times \mathbb{R}\to \mathbb{R}$, defined by
\begin{equation}
\varphi(x,t) =
\begin{cases}
a(x,|t|)t & \text{ for } t \ne 0, \\
0, & \text{ for } t = 0,
\end{cases}
\end{equation}
and for all $x\in \Omega$, $\varphi(x,\cdot):\mathbb{R}\to \mathbb{R}$ is an odd, increasing homeomorphism. For the function $\varphi$ above, if we define
\begin{equation}\label{e2.2}
\Phi(x,t) = \int_0^t \varphi(x,s)ds, \,\, \forall x \in \Omega, \, t\geq 0
\end{equation}
then the function $\Phi:\Omega\times [0,+\infty) \rightarrow [0,+\infty)$ is called a \textit{generalized} $N$-\textit{function} if it satisfies the following conditions (see e.g., \cite{Adams,Musielak,RenRao}):
\begin{itemize}
\item [$(\Phi_{0})$]for almost all $x\in\Omega$, $\Phi(x,\cdot)$ is a $N$-\textit{function}, i.e., convex,  nondecreasing and continuous function of $t$ such that, $\Phi(x,0)=0$, $\Phi(x,t)>0$ for all $t>0$, and
      \begin{equation*}
        \lim_{t\rightarrow 0}\frac{\Phi(x,t)}{t}=0, \qquad \lim_{t\rightarrow\infty}\frac{\Phi(x,t)}{t}=+\infty
      \end{equation*}
\item [$(\Phi_{1})$] $\Phi(\cdot,t)$ is a measurable function  on $\Omega$ for all $t\geq0$.
\end{itemize}
The set of all generalized $N$-functions is denoted by $N(\Omega)$.
The function $\bar{\Phi}$ defined by
\begin{equation}
\bar{\Phi}(x,t) = \int_0^t \varphi^{-1}(x,s)ds, \,\, \forall x \in \Omega, t\geq 0
\end{equation}
is called the \textit{complementary (or conjugate) function} to $\Phi$, where $\bar{\Phi}$ satisfies the following
\begin{equation*}
\bar{\Phi}(x,t) = \sup_{s>0}\{st-\Phi(x,s):~ s\in \mathbb{R}\}, ~\forall x\in \Omega, t \geq 0
\end{equation*}

It is well known that $\bar{\Phi}\in N(\Omega)$, and then the following Young inequality holds
\begin{equation}
  st\leq \Phi(x,t)+\bar{\Phi}(x,s) ~ \text{for}~ x\in \Omega ~\text{and}~ t,s\in \mathbb{R}.
\end{equation}

The function $\Phi$ allow us to define the \textit{Musielak-Sobolev spaces}, also called \textit{the generalized Orlicz spaces}, by
\begin{equation*}
  L^{\Phi}(\Omega)=\{u:\Omega\rightarrow\mathbb{R}~\text{is measurable}; \exists\lambda>0 ~\text{such that} ~ \int_{\Omega}\Phi(x,|u(x)|/\lambda)dx<+\infty\}.
\end{equation*}
Moreover, by $\Delta_{2}$-\textit{condition} (see below), $L^{\bar{\Phi}}(\Omega)$ is the dual space of $L^{\Phi}(\Omega)$, i.e., $(L^{\Phi}(\Omega))^{*}=L^{\bar{\Phi}}(\Omega)$.\\
In the sequel, we also use the following assumptions for $\Phi$:
\begin{equation}
1< \varphi_0 := \inf_{t>0}\frac{t\varphi(x,t)}{\Phi(x,t)} \leq \frac{t\varphi(x,t)}{\Phi(x,t)} \leq \varphi^0 := \sup_{t>0}\frac{t\varphi(x,t)}{\Phi(x,t)} < \infty,  ~\forall x\in \Omega, t \geq 0
\end{equation}
\begin{equation}
\einf_{x\in\Omega}\Phi(x,t)>0, ~\forall t> 0
\end{equation}
\begin{equation}
\text{the function}\,\,\, t\rightarrow \Phi(x,\sqrt{t})\,\,\, \text{is convex, } ~\forall x\in \Omega, t \geq 0
\end{equation}
By help of assumption (2.5), the Musielak-Sobolev spaces coincides the equivalence classes of measurable functions $u:\Omega\rightarrow \mathbb{R}$ such that
\begin{equation}\label{e2.4}
\int_{\Omega}\Phi(x,|u(x)|)dx < \infty
\end{equation}
and is equipped with the Luxembourg norm
\begin{equation}\label{e2.5}
|u|_\Phi : = \inf\left\{\mu>0 : ~ \int_\Omega \Phi(x,|u(x)|/\lambda)dx\leq 1\right\}
\end{equation}
For the Musielak-Orlicz spaces, H\"older inequality reads as follows (see \cite{Adams},\cite{RenRao})
\begin{equation*}
\int_\Omega uv\,dx \leq 2 \|u\|_{L^\Phi(\Omega)}\|v\|_{L^{\bar{\Phi}}(\Omega)}
\quad \text{ for all } u \in L^\Phi(\Omega) \text{ and } v \in L^{\bar{\Phi}}(\Omega)
\end{equation*}
The {\em Musielak-Sobolev spaces} $W^{1,\Phi}(\Omega)$ is the space defined by
\begin{equation*}
W^{1,\Phi}(\Omega):= \left\{u \in L^\Phi(\Omega):~ \frac{\partial u}{\partial
x_i} \in L^\Phi (\Omega), ~ i = 1, 2,..., N\right\}
\end{equation*}
under the norm
\begin{equation}\label{e2.6}
\|u\|_{1,\Phi} : = |u|_\Phi +|\nabla u|_\Phi
\end{equation}
Now we introduce {\em Musielak-Sobolev spaces with zero boundary traces} $W_0^{1,\Phi}(\Omega)$  as the closure of $C_{0}^{\infty}(\Omega)$ in $W^{1,\Phi}(\Omega)$ under the norm $\|u\|_{1,\Phi}$. Moreover, by help of the well-known \textit{Poincar\'e inequality}, we can define an equivalent norm $\|\cdot\|_{\Phi}$ on $W_0^{1,\Phi}(\Omega)$ by
\begin{equation}
\|u\|_{\Phi} : = |\nabla u|_\Phi
\end{equation}
\begin{remark}
\begin{enumerate}
\item For the case $\Phi(x,t):=\Phi(t)$, we obtain $L^\Phi(\Omega)$ and $W^{1,\Phi}(\Omega)$ called \textit{Orlicz spaces} and \textit{Orlicz-Sobolev spaces}, respectively (see \cite{Krasnosels'kii,Musielak,RenRao}).
\item For the case $\Phi(x,t):=|t|^{p(x)}$, where $p(x)$ is a continuous function on $\overline{\Omega}$ with $p(x)>1$, we replace $L^\Phi(\Omega)$ by $L^{p(x)}(\Omega)$ and $W^{1,\Phi}(\Omega)$ by $W^{1,p(x)}(\Omega)$ and call them \textit{variable exponent Lebesgue spaces} and \textit{variable exponent Sobolev spaces}, respectively (see \cite{Adams,DCU,Diening}).
\end{enumerate}
\end{remark}
\begin{proposition}[\protect\cite{Adams}] If (2.5)-(2.7) hold then the spaces $L^\Phi(\Omega)$ and $W^{1,\Phi}(\Omega)$ are separable and reflexive Banach spaces.
\end{proposition}

\begin{proposition}[\protect\cite{XLFan1,MihaRadu}]
\label{pro2.1} Let define the modular $\rho(u):=\int_{\Omega}\Phi(x,|\nabla u|)dx:W_0^{1,\Phi}(\Omega)\rightarrow \mathbb{R}$. Then for every $u_{n},u\in W_0^{1,\varphi}(\Omega)$, we have

\begin{itemize}
\item[$(i)$] $\|u\|_{\Phi}^{\varphi^0} \leq \rho(u)\leq \|u\|_{\Phi}^{\varphi_0}$ \, if \, $\|u\|_{\Phi}<1$
\item[$(ii)$] $\|u\|_{\Phi}^{\varphi_0} \leq \rho(u)\leq \|u\|_{\Phi}^{\varphi^0}$ \, if \, $\|u\|_{\Phi}>1$
\item[$(iii)$] $\|u\|_{\Phi} \leq \rho(u)+1$
\item[$(iv)$] $\|u_{n}-u\|_{\Phi}\rightarrow 0 \Leftrightarrow \rho(u_{n}-u)\rightarrow 0$
\item[$(v)$] $\|u_{n}-u\|_{\Phi}\rightarrow \infty \Leftrightarrow \rho(u_{n}-u)\rightarrow \infty$
\end{itemize}
\end{proposition}
We want to remark that Proposition 2.3 $(iv)-(v)$ mean that norm and modular topology coincide on $L^{\Phi }(\Omega )$ provided $\Phi$ satisfies (2.5), which enables that well-known $\Delta_{2}$-\textit{condition} holds, i.e.,
\begin{equation}
\Phi(x,2t)\leq M\Phi(x,t), \,\,\, \text{for all}\,\,\, x\in \Omega, t \geq 0
\end{equation}
where $M$ is a positive constant (see \cite{MihaRadu}).\\
Furthermore, if $\Psi,\Phi \in N(\Omega)$ and
\begin{equation}
  \Psi(x,t)\leq k_{1}\Phi(x,k_{2}t)+h(x), \,\,\, \text{for all}\,\,\, x\in \Omega, t \geq 0
\end{equation}
holds, where $h \in L^{1}(\Omega)$ with $h(x)\geq 0$ a.e. $x\in \Omega$, $k_{1},k_{2}$ are positive constants, then we have the following continuous embeddings (see \cite{Musielak}):
\begin{itemize}
  \item [$(i)$] $L^{\Phi }(\Omega )\hookrightarrow L^{\Psi }(\Omega )$
  \item [$(ii)$] $W^{1,\Phi }(\Omega )\hookrightarrow W^{1,\Psi }(\Omega )$
\end{itemize}

We also assume that the following condition hold for function $\Phi$.\\
For every $t > 0$ there exists a constant $C_{t} > 0$ such that
\begin{itemize}
  \item [$(\Phi_{3})$] $C_{t}\leq \Phi(x,t) \leq C^{-1}_{t}$
\end{itemize}
for a.e. $x\in \Omega$.
\begin{proposition}[\protect\cite{Fan-D}] Assume that $\Omega$ is a bounded domain with smooth boundary $\partial\Omega$. Then the embedding $W^{1,p(x)}(\Omega)\hookrightarrow L^{r(x)}(\Omega)$ is compact provided $r, p \in C(\overline{\Omega})$ such that $p^{-}>1$, $1\leq r(x)<p^{*}(x)$, where $p^{*}(x):=\frac{Np(x)}{N-p(x)}$ if $p(x)<N$ and $p^{*}(x):=+\infty$ if $p(x)\geq N$.
\end{proposition}
\begin{remark} First, we note that for $t>1$ and $s>0$ it holds $t^{\varphi_{0}}\Phi(x,s)\leq\Phi(x,ts)\leq t^{\varphi^{0}}\Phi(x,s)$. Indeed, from the assumption (2.5), we have
\[
\varphi_{0}\leq\frac{z\varphi(x,z)}{\Phi(x,z)}\leq\varphi^{0}, \,\,\, \forall x\in \Omega, z \geq 0
\]
Considering that for almost all $x\in\Omega$, $\Phi(x,z)$ is a convex, nondecreasing and continuous function of $z$, we can proceed as follows
\[
\int^{ts}_{s}\frac{\varphi_{0}}{z}dz\leq\int^{ts}_{s} \frac{\varphi(x,z)}{\Phi(x,z)}\leq\int^{ts}_{s}\frac{\varphi^{0}}{z}dz
\]
and hence
\begin{equation}
t^{\varphi_{0}}\Phi(x,s)\leq\Phi(x,ts)\leq t^{\varphi^{0}}\Phi(x,s)
\end{equation}
Now, if we consider $(\Phi_{3})$ and the inequality (2.14) together, we can obtain
\begin{equation}
C_{t}t^{\varphi_{0}}\leq \Phi(x,st)+C
\end{equation}
Hence, if we consider (2.15) along with (2.13) where $\frac{1}{k_{1}}=C_{t}$, $k_{2}=s$ and $h(x)=C=const\geq0$, the Musielak-Sobolev space $W^{1,\Phi}(\Omega)$ is continuously embedded in the variable Sobolev space $W^{1,\varphi_0 }(\Omega)$. On the other hand, $W^{1,\varphi_0 }(\Omega)$ is compactly embedded in the variable Lebesgue space $L^{r(x)}(\Omega)$ for all $1\leq r(x)<\varphi_{0}^{*}:=\frac{N\varphi_{0}}{N-\varphi_{0}}$ with $r\in C(\overline{\Omega })$. As a result, $W^{1,\Phi}(\Omega)$ is continuously and compactly embedded in the variable Lebesgue space $L^{r(x)}(\Omega)$.
\end{remark}
\begin{remark}
The functional $\rho$ is from $C^{1}(W_{0}^{1,\Phi}(\Omega),\mathbb{R})$ with the derivative
$$
\langle\rho^{\prime}(u),v\rangle=\int_{\Omega}a(x,|\nabla u|)\nabla u\cdot\nabla vdx
$$
where $\langle \cdot, \cdot\rangle$ is the dual pairing between $W_{0}^{1,\Phi}(\Omega)$ and its dual $(W_{0}^{1,\Phi}(\Omega))^{*}$(see \cite{MihaRadu}).
\end{remark}

The following Proposition generalizes the definition of convexity, and therefore, we give a proof for the convenience.

\begin{proposition}
Let $X$ be a vector space and let $I:X \rightarrow \mathbb{R}$. Then $I$ is convex if and only if
\begin{equation}
I((1-\lambda)u+\lambda v)< (1-\lambda)\theta+\lambda \beta, \,\,\,\, 0<\lambda<1
\end{equation}
whenever $I(u)<\theta$ and $I(v)<\beta$, for all $u,v \in X$ and $\theta,\beta \in \mathbb{R}$.
\end{proposition}
\begin{proof} Assume that functional $I:X \rightarrow \mathbb{R}$ is convex. Moreover, since $I$ is a real-valued functional, there are real numbers $\theta,\beta \in \mathbb{R}$ such that $I(u)<\theta$ and $I(v)<\beta$. Then
\begin{equation*}
I((1-\lambda)u+\lambda v)<(1-\lambda)I(u)+\lambda I(v) <(1-\lambda)\theta+\lambda \beta, \,\,\,\, 0<\lambda<1.
\end{equation*}
On the other hand, assume that (2.16) holds. Since $I(u)<\theta$ and $I(v)<\beta$, we can write, for all $\varepsilon >0$,
$$
I(u)<I(u)+\varepsilon:=\theta
$$
$$
I(v)<I(v)+\varepsilon:=\beta
$$
Therefore,
\begin{equation}
I((1-\lambda)u+\lambda v)<(1-\lambda)I(u)+\lambda I(v)+\varepsilon, \,\,\,\, 0<\lambda<1
\end{equation}
If we consider that (2.17) holds for any $\varepsilon >0$, we conclude
\begin{equation*}
I((1-\lambda)u+\lambda v)\leq(1-\lambda)I(u)+\lambda I(v)
\end{equation*}
\end{proof}
\section{The main results}
\begin{theorem} Suppose that the following assumptions hold:

\begin{itemize}

\item [(G0)] $g(x) \in C^{1}(\overline{\Omega})$ is a nontrivial nonnegative function.

\item [(A0)] $\mathcal{A}:(0,\infty)\rightarrow (0,\infty)$ is a continuous function and satisfies the growth condition
\begin{equation*}
m_{1}t^{\alpha-1}\leq\mathcal{A}(t) \leq m_{2}t^{\alpha-1}
\end{equation*}
where $m_{1},m_{2},\alpha$ are real numbers such that $m_{2} \geq m_{1}>1$ and $\alpha >1$.
\end{itemize}
Then problem \ref{e1.1} has a positive ground state solution in $W_0^{1,\Phi}(\Omega)$ with a negative energy level.
\end{theorem}
We define the functional $J:W_0^{1,\Phi}(\Omega)\rightarrow \mathbb{R}$ corresponding to problem \ref{e1.1} by
$$
J(u)= \widehat{\mathcal{A}}\left(\int_{\Omega}\Phi(x,|\nabla u|)dx\right)-\int_{\Omega}\frac{g(x)|u|^{1-\gamma(x)}}{1-\gamma(x)}dx
$$
where $\widehat{\mathcal{A}}(t)=\int^{t}_{0}\mathcal{A}(s)ds$.
\begin{definition} A function $u$ is called a weak solution to problem \ref{e1.1} if $u\in W_0^{1,\Phi}(\Omega)$ such that $u>0$ in $\Omega$ and
\begin{equation}
\mathcal{A}\left(\int_{\Omega}\Phi(x,|\nabla u|)dx\right)\int_{\Omega}a(x,|\nabla u|)\nabla u \cdot \nabla v dx=\int_{\Omega}g(x)u^{-\gamma(x)}vdx
\end{equation}
for all $v\in W_0^{1,\Phi}(\Omega)$.
\end{definition}
We would like to notice that due to singular term, the derivative operator $J^{\prime}$ is not continuous on $W_0^{1,\Phi}(\Omega)$, that is, $J$ is not Fr\'{e}chet differentiable on $W_0^{1,\Phi}(\Omega)$. Therefore, we must show that any global minimizer is in fact a solution to problem \ref{e1.1}. To this end, to obtain the main result given in Theorem 3.1, it is necessary to show that Lemma 3.3 holds.
\begin{lemma} The functional $J$ attains the global minimizer in $W_0^{1,\Phi}(\Omega)$, that is, there exists a function $u_{*}\in W_0^{1,\Phi}(\Omega)$ such that
\begin{equation}
m=J(u_{*})=\inf_{u\in W_0^{1,\Phi}(\Omega)}J(u)<0
\end{equation}
\end{lemma}
\begin{proof}
By (G0), (A0), H\"{o}lder inequality, Proposition 2.3 and the continuous embeddings $W_0^{1,\Phi}(\Omega)\hookrightarrow L^{p(x)}(\Omega)$ and $W_0^{1,\Phi}(\Omega)\hookrightarrow L^{\frac{p(x)}{p(x)+\gamma(x)-1}}(\Omega)$, it follows
\begin{align*}
  |J(u)| & \leq \frac{m_{2}}{\alpha}\left(\int_{\Omega}\Phi(x,|\nabla u|)dx\right)^{\alpha}+\frac{|g|_{\infty}}{1-\gamma^{+}}\int_{\Omega}|u|^{1-\gamma(x)}dx\\
  & \leq
 \frac{m_{2}}{\alpha}\|u\|_{\Phi}^{\alpha\varphi^{0}}+\frac{|g|_{\infty}}{1-\gamma^{+}}||u|^{1-\gamma(x)}|_{L^{\frac{p(x)}{1-\gamma(x)}}
  (\Omega)}|\mathbf{1}|_{L^{\frac{p(x)}{p(x)+\gamma(x)-1}}(\Omega)}\\
  & \leq \frac{m_{2}}{\alpha}\|u\|_{\Phi}^{\alpha\varphi^{0}}+\frac{c|g|_{\infty}}{1-\gamma^{+}}\|u\|^{1-\gamma^{-}}_{\Phi}<+\infty
\end{align*}
which shows that $J$ is well-defined on $W_0^{1,\Phi}(\Omega)$.\\
Denote $K:W_0^{1,\Phi}(\Omega)\rightarrow \mathbb{R}$ by $K(u):=\widehat{\mathcal{A}}(\rho(u))$.
Considering the fact that the functional $\rho$ is of class $C^{1}(W^{1,\Phi}(\Omega),\mathbb{R})$ (see Remark 2.6), and $\widehat{\mathcal{A}}$ is a continuous function, it is easy to see that the composition functional $
K$ is continuous on $W_0^{1,\Phi}(\Omega)$. Further, by the well-known inequality
$$
|a^{p}-b^{p}|\leq |a-b|^{p},\,\,\, \text{for any real numbers}\,\, a,b\geq 0\,\, \text{and} \,\,0<p<1,
$$
we obtain
\begin{align*}
  |J(u)-J(v)| & \leq\left| \widehat{\mathcal{A}}(\rho(u))-\widehat{\mathcal{A}}(\rho(v))\right|+\frac{|g|_{\infty}}{1-\gamma^{+}}\int_{\Omega}||u|^{1-\gamma(x)}-|v|^{1-\gamma(x)}|dx\\
  & \leq |K(u)-K(v)|+\frac{|g|_{\infty}}{1-\gamma^{+}}\int_{\Omega}|u-v|^{1-\gamma(x)}dx\\
  & \leq |K(u)-K(v)|+\frac{|g|_{\infty}}{1-\gamma^{+}}||u-v|^{1-\gamma(x)}|_{L^{\frac{p(x)}{1-\gamma(x)}}
  (\Omega)}|\mathbf{1}|_{L^{\frac{p(x)}{p(x)+\gamma(x)-1}}(\Omega)}\\
   & \leq |K(u)-K(v)|+\frac{c|g|_{\infty}}{1-\gamma^{+}}\|u-v\|^{1-\gamma^{-}}_{\Phi}
\end{align*}
for any $u,v \in W_0^{1,\Phi}(\Omega)$. Therefore $J$ is continuous on $W_0^{1,\Phi}(\Omega)$.\\
Let $u\in W_0^{1,\Phi}(\Omega)$. Then, applying the same steps as we did above, it follows
\begin{align}
  J(u) &  \geq \frac{m_{1}}{\alpha}\left(\int_{\Omega}\Phi(x,|\nabla u|)dx\right)^{\alpha}-\frac{|g|_{\infty}}{1-\gamma^{+}}||u|^{1-\gamma(x)}|_{L^{\frac{p(x)}{1-\gamma(x)}}
  (\Omega)}|\mathbf{1}|_{L^{\frac{p(x)}{p(x)+\gamma(x)-1}}(\Omega)} \nonumber\\
  & \geq \frac{m_{1}}{\alpha}\|u\|_{\Phi}^{\alpha\varphi_{0}}-c\|u\|^{1-\gamma^{-}}_{\Phi}
\end{align}
Since $\alpha\varphi_{0}>1-\gamma^{+}$, $J$ is coercive, namely, $J(u)\rightarrow +\infty$ as $\|u\|_{\Phi}\rightarrow\infty$, and bounded below.\\
Now, we shall show that $J$ is convex on $W_{0}^{1,\Phi}(\Omega)$. To this end, using (A0) and considering the assumptions for $g$ and $\gamma$, we have
\begin{align*}
J(u)& = \widehat{\mathcal{A}}\left(\int_{\Omega}\Phi(x,|\nabla u|)dx\right)-\int_{\Omega}\frac{g(x)|u|^{1-\gamma(x)}}{1-\gamma(x)}dx\\
  &\leq \frac{m_{2}}{\alpha}\max\{\|u\|_{\Phi}^{\alpha\varphi_{0}},\|u\|_{\Phi}^{\alpha\varphi^{0}}\}:=\theta
\end{align*}
and
\begin{align*}
J(v)& = \widehat{\mathcal{A}}\left(\int_{\Omega}\Phi(x,|\nabla v|)dx\right)-\int_{\Omega}\frac{g(x)|v|^{1-\gamma(x)}}{1-\gamma(x)}dx\\
  &\leq \frac{m_{2}}{\alpha}\max\{\|v\|_{\Phi}^{\alpha\varphi_{0}},\|v\|_{\Phi}^{\alpha\varphi^{0}}\}:=\beta
\end{align*}
for all $u,v \in W_0^{1,\Phi}(\Omega)$. Since $\Phi$ is convex, so is $\Theta(\cdot)=\int_{\Omega}\Phi(x,|\nabla \cdot|)dx$ (see \cite{Diening}).\\
Since $\widehat{\mathcal{A}}$ satisfies (A0), it is a continuous and monotone function on $(0,+\infty)$, and hence, its convexity follows. Therefore, for $0<\lambda<1$, we have
$$
\widehat{\mathcal{A}}(\Theta((1-\lambda)u+\lambda v))\leq (1-\lambda)\widehat{\mathcal{A}}(\Theta)+\lambda\widehat{\mathcal{A}}(\Theta)
$$
Therefore, considering the all pieces of information obtained above along with (A0) and (F1), it follows
\begin{align*}
J((1-\lambda)u+\lambda v))& = \widehat{\mathcal{A}}\left(\int_{\Omega}\Phi(x,|\nabla ((1-\lambda)u+\lambda v)|)dx\right)-\int_{\Omega}\frac{g(x)|(1-\lambda)u+\lambda v|^{1-\gamma(x)}}{1-\gamma(x)}dx\\
  &\leq (1-\lambda)\frac{m_{2}}{\alpha}\max\{\|u\|_{\Phi}^{\alpha\varphi_{0}},\|u\|_{\Phi}^{\alpha\varphi^{0}}\}
  +\lambda\frac{m_{2}}{\alpha}\max\{\|v\|_{\Phi}^{\alpha\varphi_{0}},\|v\|_{\Phi}^{\alpha\varphi^{0}}\}\\
  &\leq (1-\lambda)\theta+\lambda\beta
\end{align*}
Hence, by Proposition 2.7, $J$ is convex on $W_{0}^{1,\Phi}(\Omega)$.\\
As the functional $J$ is continuous, coercive and convex, it has a global minimum belonging to $W_{0}^{1,\Phi}(\Omega)$, which in turn becomes a solution to problem \ref{e1.1}.\\ Let us denote $$m=\inf_{u\in W_0^{1,\Phi}(\Omega)}J(u)$$ which is well-defined due to (3.3).\\
Now, applying the same arguments used in Remark 2.5, we can obtain that for $t>0$ small enough and $s>0$, it holds $\Phi(x,ts)\leq t^{\varphi_{0}}\Phi(x,s)$. Indeed, from assumption (2.5) and the properties of $\Phi(x,z)$, we can proceed as follows
\[
\int^{s}_{ts}\frac{\varphi_{0}}{z}dz \leq\int^{s}_{ts}\frac{\varphi(x,z)}{\Phi(x,z)}dz
\]
\[
\Phi(x,ts)\leq t^{\varphi_{0}}\Phi(x,s)
\]
For $0\neq \phi \in W_{0}^{1,\Phi}(\Omega)$ and $0<t\in\mathbb{R}$ small enough, it reads
\begin{align*}
J(t\phi)& = \widehat{\mathcal{A}}\left(\int_{\Omega}\Phi(x,|\nabla t\phi|)dx\right)-\int_{\Omega}\frac{g(x)|t\phi|^{1-\gamma(x)}}{1-\gamma(x)}dx\\
 &\leq \frac{m_{2}}{\alpha}t^{\alpha\varphi_{0}}\left(\int_{\Omega}\Phi(x,|\nabla \phi|)dx\right)^{\alpha}- \frac{t^{1-\gamma^{+}}}{1-\gamma^{+}}\int_{\Omega}g(x)|\phi|^{1-\gamma(x)}dx
\end{align*}
Since $1-\gamma^{+}<\alpha\varphi_{0}$, we obtain that $J(t\phi)< 0 $. If we set $t\phi=u$ with $\|u\|_{\Phi}<1$, we obtain that $m=\inf_{u\in W_0^{1,\Phi}(\Omega)}J(u)<0$. On the other hand, if we take into account the definition of $m$, there exists a minimizing sequence $(u_{n})$ of $W_{0}^{1,\Phi}(\Omega)$ such that
\begin{equation}
m=\lim_{n\rightarrow\infty}J(u_{n})<0
\end{equation}
Moreover, since $J(u_{n})=J(|u_{n}|)$ we may assume that $u_{n}\geq 0$. Due to the coercivity of $J$, $(u_{n})$ must be bounded in $W_{0}^{1,\Phi}(\Omega)$ other wise we would have that $J(u_{n})\rightarrow +\infty$ as $\|u_{n}\|_{\Phi}\rightarrow\infty$ which contradicts (3.4). Since $W_{0}^{1,\Phi}(\Omega)$ is reflexive there exists a subsequence, not relabelled, and $u_{*}\in W_{0}^{1,\Phi}(\Omega)$ such that

$u_{n}\rightharpoonup u_{*}$ in $W_{0}^{1,\Phi}(\Omega)$,

$u_{n}\rightarrow u_{*}$ in $L^{s(x)}(\Omega)$,\,\, $1\leq s(x)<p^{*}(x)$

$u_{n}(x) \rightarrow u_{*}(x) $ a.e. in $\Omega$.\\
Since $J$ is continuous and convex on $W_{0}^{1,\Phi}(\Omega)$, it is weakly lower semi-continuous on $W_{0}^{1,\Phi}(\Omega)$. Therefore,
\begin{align}
  m\leq J(u_{*}) & =\widehat{\mathcal{A}}\left(\int_{\Omega}\Phi(x,|\nabla u_{*}|)dx\right)-\int_{\Omega}\frac{g(x)|u_{*}|^{1-\gamma(x)}}{1-\gamma(x)}dx\\
    & \leq \liminf_{n\rightarrow\infty} J(u_{n})=m
\end{align}
which means
\begin{equation}
m=J(u_{*})=\inf_{W_0^{1,\Phi}(\Omega)}J(u)<0
\end{equation}
\end{proof}
\begin{proof}(\textbf{Proof of Theorem 3.1}) Since $m=J(u_{*})<0=J(0)$, it must be $u_{*}\geq 0$, $u_{*}\neq0$. For $\phi\in W_{0}^{1,\Phi}(\Omega)$, $\phi\geq 0$ and $t>0$, we have
\begin{align*}
  0\leq & \liminf_{t\rightarrow 0}\frac{J(u_{*}+t\phi)-J(u_{*})}{t}\\
  & \leq \mathcal{A}(\rho(u_{*}))\int_{\Omega}a(x,|\nabla u_{*}|)\nabla u_{*} \cdot \nabla \phi dx- \limsup_{t\rightarrow 0}\int_{\Omega}g(x)\frac{(u_{*}+t\phi)^{1-\gamma(x)}-u_{*}^{1-\gamma(x)}}{1-\gamma(x)}dx
\end{align*}
or
\begin{align}
 \limsup_{t\rightarrow 0}\int_{\Omega}g(x)\frac{(u_{*}+t\phi)^{1-\gamma(x)}-u_{*}^{1-\gamma(x)}}{1-\gamma(x)}dx & \leq \mathcal{A}(\rho(u_{*}))\int_{\Omega} a(x,|\nabla u_{*}|)\nabla u_{*} \cdot \nabla \phi dx
\end{align}
By the mean value theorem, there exists $\theta\in (0,1)$ such that
\begin{align}
  \int_{\Omega}g(x)\frac{(u_{*}+t\phi)^{1-\gamma(x)}-u_{*}^{1-\gamma(x)}}{1-\gamma(x)}dx = &  \int_{\Omega}g(x)(u_{*}+t\theta\phi)^{-\gamma(x)}\phi dx
\end{align}
On the other hand, since we have
$$
(u_{*}+t\theta\phi)^{-\gamma(x)}\phi \geq 0,\,\,\, \forall x \in \Omega
$$
and
\begin{equation*}
(u_{*}+t\theta\phi)^{-\gamma(x)}\phi\rightarrow u_{*}^{-\gamma(x)}\phi,\,\, as\,\, t\rightarrow 0,\, a.e.\, x \in \Omega
\end{equation*}
we can apply Fatou's lemma to (3.9), that is,
\begin{align}
 \limsup_{t\rightarrow 0}\int_{\Omega}g(x)\frac{(u_{*}+t\phi)^{1-\gamma(x)}-u_{*}^{1-\gamma(x)}}{1-\gamma(x)}dx & \geq \liminf_{t\rightarrow 0}\int_{\Omega}g(x)\frac{(u_{*}+t\phi)^{1-\gamma(x)}-u_{*}^{1-\gamma(x)}}{1-\gamma(x)}dx \nonumber\\
 &= \liminf_{t\rightarrow 0}\int_{\Omega}g(x)(u_{*}+t\theta\phi)^{-\gamma(x)}\phi dx \nonumber\\
 & \geq \int_{\Omega}g(x)u_{*}^{-\gamma(x)}\phi dx \geq 0 \label{3.13}
\end{align}
Thus, by (3.8) and (3.10) we can write
\begin{align}
\mathcal{A}(\rho(u_{*}))\int_{\Omega} a(x,|\nabla u_{*}|)\nabla u_{*} \cdot \nabla \phi dx- \int_{\Omega}g(x)u_{*}^{-\gamma(x)}\phi dx & \geq 0,\,\,\, \forall\phi\in W_{0}^{1,\Phi}(\Omega), \phi\geq 0
\end{align}

and hence, we obtain that function $u_{*}\in W_{0}^{1,\Phi}(\Omega)$ satisfies
\begin{align}
-\mathcal{A}(\rho(u_{*}))\mathrm{div}(a(x,|\nabla u_{*}|)\nabla u_{*}) & \geq 0\,\,\, \text{in}\,\, \Omega
\end{align}
in the weak sense. Since $u_{*}\geq 0$ and $u_{*}\neq0$, by the strong maximum principle for weak solutions, we must have
$$
u_{*}(x)>0,\,\,\, \forall x \in \Omega
$$
Next, we show that $u_{*}\in W_{0}^{1,\Phi}(\Omega)$ satisfies (3.1). The proof below has been adapted from one given in \cite{Lair}. For given $\delta>0$, define $\Lambda:[-\delta,\delta]\rightarrow (-\infty,\infty)$ by $\Lambda(t)=J(u_{*}+tu_{*})$. Then $\Lambda$ achieves its minimum at $t=0$. Thus,
\begin{align*}
 \frac{d}{dt}\Lambda(t)|_{t=0}=\frac{d}{dt}J(u_{*}+tu_{*})|_{t=0}= 0
\end{align*}
or
\begin{align}
 \mathcal{A}(\rho(u_{*}))\int_{\Omega} a(x,|\nabla u_{*}|)|\nabla u_{*}|^{2}dx- \int_{\Omega}g(x)u_{*}^{1-\gamma(x)}dx=0
\end{align}
Let us take $\phi\in W_{0}^{1,\Phi}(\Omega)$, and define $\Psi\in W_{0}^{1,\Phi}(\Omega)$ such that $\Psi:=(u_{*}+\varepsilon \phi)^{+}=\max\{0,u_{*}+\varepsilon \phi\}$, $\varepsilon>0$. Clearly, $\Psi\geq 0$. If we replace $\Psi$ both in (3.11) and (3.13), we have
\begin{align}
 0&\leq \mathcal{A}(\rho(u_{*}))\int_{\{u_{*}+\varepsilon \phi\geq 0\}} a(x,|\nabla u_{*}|)\nabla u_{*} \cdot \nabla (u_{*}+\varepsilon \phi)dx-\int_{\{u_{*}+\varepsilon \phi\geq 0\}}g(x)u_{*}^{-\gamma(x)}(u_{*}+\varepsilon \phi) dx \nonumber\\
 &=\mathcal{A}(\rho(u_{*}))\left(\int_{\Omega}-\int_{\{u_{*}+\varepsilon \phi< 0\}}\right) a(x,|\nabla u_{*}|)\nabla u_{*} \cdot \nabla (u_{*}+\varepsilon \phi)dx\nonumber\\
 &-\left(\int_{\Omega}-\int_{\{u_{*}+\varepsilon \phi< 0\}}\right)g(x)u_{*}^{-\gamma(x)}(u_{*}+\varepsilon \phi) dx \nonumber\\
 &=\mathcal{A}(\rho(u_{*}))\int_{\Omega} a(x,|\nabla u_{*}|)|\nabla u_{*}|^{2}dx- \int_{\Omega}g(x)u_{*}^{1-\gamma(x)}dx \nonumber\\
 &+\varepsilon\mathcal{A}(\rho(u_{*}))\int_{\Omega} a(x,|\nabla u_{*}|)\nabla u_{*} \cdot \nabla \phi dx- \varepsilon\int_{\Omega}g(x)u_{*}^{-\gamma(x)}\phi dx \nonumber\\
 &-\mathcal{A}(\rho(u_{*}))\int_{\{u_{*}+\varepsilon \phi< 0\}} a(x,|\nabla u_{*}|)\nabla u_{*} \cdot \nabla (u_{*}+\varepsilon \phi)dx+\int_{\{u_{*}+\varepsilon \phi< 0\}}g(x)u_{*}^{-\gamma(x)}(u_{*}+\varepsilon \phi) dx \nonumber\\
 &=\varepsilon\left(\mathcal{A}(\rho(u_{*}))\int_{\Omega} a(x,|\nabla u_{*}|)\nabla u_{*} \cdot \nabla \phi dx-\int_{\Omega}g(x)u_{*}^{-\gamma(x)}\phi dx\right)\nonumber\\
 &-\mathcal{A}(\rho(u_{*}))\int_{\{u_{*}+\varepsilon \phi< 0\}} a(x,|\nabla u_{*}|)\nabla u_{*} \cdot \nabla (u_{*}+\varepsilon \phi)dx+\int_{\{u_{*}+\varepsilon \phi< 0\}}g(x)u_{*}^{-\gamma(x)}(u_{*}+\varepsilon \phi) dx \\
 &\leq \varepsilon\left(\mathcal{A}(\rho(u_{*}))\int_{\Omega} a(x,|\nabla u_{*}|)\nabla u_{*} \cdot \nabla \phi dx-\int_{\Omega}g(x)u_{*}^{-\gamma(x)}\phi dx\right) \\
 &-\varepsilon\mathcal{A}(\rho(u_{*}))\int_{\{u_{*}+\varepsilon \phi< 0\}} a(x,|\nabla u_{*}|)\nabla u_{*} \cdot \nabla \phi dx \nonumber
\end{align}
Considering that $u_{*}>0$ and Lebesgue measure of the domain of integration $\{u_{*}+\varepsilon \phi<0\}$ tends to zero as $\varepsilon \rightarrow 0$, and (A0) it reads
$$
\mathcal{A}(\rho(u_{*}))\int_{\{u_{*}+\varepsilon \phi< 0\}} a(x,|\nabla u_{*}|)\nabla u_{*} \cdot \nabla \phi dx \rightarrow 0,\,\,\, \text{as} \, \varepsilon \rightarrow 0
$$
Moreover, considering that $a(x,\cdot)\in(0,\infty)$ and (A0), we can drop the term $$-\mathcal{A}(\rho(u_{*}))\int_{\{u_{*}+\varepsilon \phi< 0\}} a(x,|\nabla u_{*}|)|\nabla u_{*}|^{2}dx$$ in (3.14) since it is negative.
Therefore, dividing (3.15) by $\varepsilon$ and letting $\varepsilon \rightarrow 0$, we obtain
\begin{align}
\mathcal{A}(\rho(u_{*}))\int_{\Omega} a(x,|\nabla u_{*}|)\nabla u_{*} \cdot \nabla \phi dx-\int_{\Omega}g(x)u_{*}^{-\gamma(x)}\phi dx
 & \geq 0
\end{align}
Considering that $\phi\in W_{0}^{1,\Phi}(\Omega)$ is arbitrary, (3.16) holds for $-\phi$ as well. As a conclusion, we obtain
\begin{align}
\mathcal{A}(\rho(u_{*}))\int_{\Omega} a(x,|\nabla u_{*}|)\nabla u_{*} \cdot \nabla \phi dx-\int_{\Omega}g(x)u_{*}^{-\gamma(x)}\phi dx
 &= 0
\end{align}
that is to say, $u_{*}\in W_{0}^{1,\Phi}(\Omega)$ is a weak solution to problem \ref{e1.1}. Additionally, since $J$ is coercive and bounded below on $W_{0}^{1,\Phi}(\Omega)$, $u_{*}$ is a positive ground state solution to problem \ref{e1.1}, i.e., a solution with minimum action among all nontrivial solutions. Additionally, since $J(u_{*})<0$ this solution has a negative energy level.
\end{proof}
\begin{theorem} Suppose the conditions of Theorem 3.1 hold. Additionally, assume the following conditions hold:
\begin{itemize}
  \item [(A1)] $\mathcal{A}$ is bounded on $(0,\infty)$, i.e., for any $t\in(0,\infty)$, there are real numbers $\underline{c},\overline{c}>0$ such that $\underline{c}\leq \mathcal{A}(t)\leq \overline{c}$.
  \item [(a1)] There exists a real number $\underline{a}>0$ such that $a(x,t)\geq \underline{a}>0$ holds for any $t\in \mathbb{R}$.
\end{itemize}
Then $u_{*}\in W_{0}^{1,\Phi}(\Omega)$ is the unique solution to problem \ref{e1.1}.
\end{theorem}
\begin{proof} Let us assume $\upsilon_{*}$ is an another solution to problem \ref{e1.1}. Then, from (3.1), we have
\begin{align*}
 & \mathcal{A}\left(\int_{\Omega}\Phi(x,|\nabla u_{*}|)dx\right)\int_{\Omega}a(x,|\nabla u_{*}|)\nabla u_{*} \cdot \nabla (u_{*}-\upsilon_{*}) dx-\int_{\Omega}g(x)u_{*}^{-\gamma(x)}(u_{*}-\upsilon_{*}) dx \\
 &-\mathcal{A}\left(\int_{\Omega}\Phi(x,|\nabla \upsilon_{*}|)dx\right)\int_{\Omega}a(x,|\nabla \upsilon_{*}|)\nabla \upsilon_{*} \cdot \nabla (u_{*}-\upsilon_{*}) dx+\int_{\Omega}g(x)\upsilon_{*}^{-\gamma(x)}(u_{*}-\upsilon_{*})dx=0
\end{align*}
or
\begin{align}
 &\int_{\Omega}\left(\mathcal{A}(\rho(u_{*}))a(x,|\nabla u_{*}|)\nabla u_{*}-\mathcal{A}(\rho(\upsilon_{*}))a(x,|\nabla \upsilon_{*}|)\nabla \upsilon_{*}\right) \cdot \nabla (u_{*}-\upsilon_{*}) dx \\
 &=\int_{\Omega}g(x)(u_{*}^{-\gamma(x)}-\upsilon_{*}^{-\gamma(x)})(u_{*}-\upsilon_{*})dx
\end{align}
For $\alpha \in (0,1)$ and $x,y \geq 0$, we have the elementary inequality
\begin{align}
  & (x^{-\alpha}-y^{-\alpha})(x-y)\leq 0
\end{align}
Moreover, by Lemma 2.4 given in \cite{Rabil}, we have the following inequality:
for any $k,l>0$, there exists a positive constant $C(\delta )$, $\delta =\min\{1,a_{0},k,l\}$, such that
\begin{equation}
(ka(|\xi |)\xi -la(|\eta |)\eta)\cdot(\xi -\eta)\geq C(\delta )\Phi (|\xi -\eta|)\quad \forall \xi ,\eta \in \mathbb{R}^{N}
\end{equation}
holds, provided that (A1) and (a1) hold. Thus, if we apply (3.21) and (3.20) to the lines (3.18) and (3.19) respectively, we obtain
\begin{align*}
0\leq &C(\delta )\int_{\Omega}\Phi(x,|\nabla u_{*}-\nabla\upsilon_{*}|) dx\leq 0
\end{align*}
or
\begin{align*}
\rho(u_{*}-\upsilon_{*})=0
\end{align*}
Therefore, by Proposition 2.3, we have
\begin{align*}
0\leq\min\{\|u_{*}-\upsilon_{*}\|^{\varphi_{0}}_{\Phi},\|u_{*}-\upsilon_{*}\|^{\varphi^{0}}_{\Phi}\} \leq \rho(u_{*}-\upsilon_{*})=0
\end{align*}
which means that
\begin{align*}
\|u_{*}-\upsilon_{*}\|_{\Phi}= 0
\end{align*}
Thus, we have $u_{*}=\upsilon_{*}$ in $W_{0}^{1,\Phi}(\Omega)$, that is, $u_{*}$ is the unique solution to problem \ref{e1.1}.
\end{proof}

\end{document}